\documentclass[a4paper,12pt]{article}
\usepackage{amsmath}
\usepackage{amsthm}
\usepackage{amssymb}
\usepackage{enumerate}

\title{On divergence of expectations of the Feynman-Kac type with 
singular potentials}

\author{Yuu Hariya\footnote{Mathematical Institute, 
Tohoku University, Aoba-ku, Sendai 980-8578, Japan. }
\thanks{Corresponding author. E-mail: hariya@math.tohoku.ac.jp}
\and Kaname Hasegawa}

\date{\empty}
\numberwithin{equation}{section}

\theoremstyle{plain}

\newtheorem{thm}{Theorem}[section]

\newtheorem{prop}{Proposition}[section]

\newtheorem{lem}{Lemma}[section]

\theoremstyle{definition}

\theoremstyle{remark}
\newtheorem{rem}{Remark}[section]

\begin{document}

\def\N {\mathbb{N}}
\def\R {\mathbb{R}}
\def\Q {\mathbb{Q}}

\def\calF {\mathcal{F}}

\def\kp {\kappa}

\def\ind {\boldsymbol{1}}

\def\al {\alpha }
\def\la {\lambda }

\def\ga {\gamma }
\def\be {\beta }

\def\br {B}

\def\W {W}

\def\pr {P}
\def\pbes {P^{(3)}}

\def\ex {E}
\def\E {\mathbb{E}}

\newcommand{\vp}{\varphi}
\newcommand{\ep}{\epsilon}

\newcommand\ND{\newcommand}
\newcommand\RD{\renewcommand}

\ND\lref[1]{Lemma~\ref{#1}}
\ND\tref[1]{Theorem~\ref{#1}}
\ND\pref[1]{Proposition~\ref{#1}}
\ND\sref[1]{Section~\ref{#1}}
\ND\ssref[1]{Subsection~\ref{#1}}
\ND\aref[1]{Appendix~\ref{#1}}
\ND\rref[1]{Remark~\ref{#1}} 
\ND\cref[1]{Corollary~\ref{#1}}
\ND\eref[1]{Example~\ref{#1}}
\ND\fref[1]{Fig.\ {#1} }
\ND\lsref[1]{Lemmas~\ref{#1}}
\ND\tsref[1]{Theorems~\ref{#1}}
\ND\dref[1]{Definition~\ref{#1}}
\ND\psref[1]{Propositions~\ref{#1}}

\ND\sub[1]{T^{\alpha}_{#1}}
\ND\ovl[1]{\overline{#1}}

\ND\qp[1]{Q_{#1}\otimes P}

\ND\btm{\beta ^{(m)}}

\ND\Xm{X^{(m)}}

\def\br {B}
\def\bd {\br '}
\def\bn {\br ^{N}}
\def\ln {L^{N}}
\def\B {{\bf B}}
\def\E {{\bf E}}
\def\P {{\bf P}}
\def\sgn {\mathrm{sgn}}
\def\O {X}
\def\I {I^{1}}
\def\II {I}

\def\vi {\nu}
\def\A {a}

\newcommand{\Cst}{C}

\def\thefootnote{{}}

\maketitle 

\begin{abstract}
Motivated by the work of Baras-Goldstein (1984), we discuss when expectations 
of the Feynman-Kac type with singular potentials are divergent. 
Underlying processes are Brownian motion and $\alpha$-stable process. 
In connection with the work of Ishige-Ishiwata (2012) concerned with 
the heat equation in the half-space with a singular potential on the boundary, 
we also discuss the same problem in the half-space for the case of 
Brownian motion. 
\footnote{{\itshape Running head}.~Divergence of expectations of the Feynman-Kac type}
\end{abstract}

\section{Introduction}\label{;intro}

For $N\ge 3$, let $V$ be a nonnegative measurable function on $\R ^{N}$  
and consider the following heat equation: 
\begin{equation}\label{;hq0}
\left\{
\begin{array}{ll}
\displaystyle{\frac{\partial}{\partial t}} u=\dfrac{1}{2}\Delta u+Vu & 
\mbox{in}\quad (0,\infty )\times \R ^N,\vspace{7pt}\\
u(0,x)=u_0(x)\ge(\not\equiv)\, 0 & \mbox{in}\quad \R ^{N}. 
\end{array}
\right.
\end{equation}
We assume $u_{0}\in C_{0}(\R ^{N})$ for simplicity. In 
\cite{bg}, Baras and Goldstein derived a sufficient condition on the potential 
function $V$ for the nonexistence of solutions to the initial value problem 
\eqref{;hq0} by using the Feynman-Kac formula. 
In the sequel we let $\vi $ be a nonnegative measurable function on $(0,\infty )$ 
that is nonincreasing near the origin. 

\begin{thm}[\cite{bg}, Theorem~6.1]\label{;t0}
Suppose that $\vi $ satisfies 
\begin{align}\label{;cond0}
  \liminf _{r\to 0+}r^2\vi (r)&>\frac{\pi ^2}{8}N^{2} 
\end{align}
and that $V$ satisfies $V(x)\ge \vi (|x|)$ for a.e.\ $x\in \R ^{N}$. 
Then for any initial datum $u_{0}$, the equation \eqref{;hq0} does not have a 
solution. 
\end{thm}

The precise meaning of the equation \eqref{;hq0} not having a solution 
will be recalled in \sref{;simpr}; in view of the Feynman-Kac formula, 
it may be regarded as the divergence of the 
expectation 
\begin{align}\label{;fkbr}
 \ex _{x}\left[ 
 u_{0}(B_{t})\exp \left( 
 \int _{0}^{t}V(B_{s})\,ds
 \right) 
 \right] 
\end{align}
for any $x\in \R ^{N}$ and $t>0$, where 
$(\{ B_{t}\} _{t\ge 0}, \{ P_{x}\} _{x\in \R ^{N}})$ is an 
$N$-dimensional Brownian motion and $\ex _{x}$ denotes the 
expectation with respect to the probability measure $P_{x}$. 

One of the objectives of the paper is to show that the condition 
\eqref{;cond0} can be relaxed as 
\begin{align}\label{;cond1}
 \liminf _{r\to 0+}r^2\vi (r)&>\frac{1}{2}j_{\frac{N-2}{2},1}^2. 
\end{align}
See \tref{;tmain1} below. Here and in the sequel, 
we denote by $j_{\mu ,1}$ the first positive zero of the Bessel 
function $J_{\mu }$ of the first kind with index $\mu $ for 
$\mu >-1$. 
Baras and Goldstein proved \tref{;t0} probabilistically, while in \cite[Theorem~2.2]{bg} 
they showed, employing an analytic approach not dependent on the Feynman-Kac 
formula, 
that in the 
case $V(x)=c/|x|^{2}$ with $c$ a positive constant, 
the number $C_{N}=\frac{1}{2}\left( \frac{N-2}{2}\right) ^{2}$ 
is the threshold for the existence and nonexistence 
of solutions to the problem; that is, for any initial datum $u_{0}\in C_{0}(\R ^{N})$, 
the equation \eqref{;hq0} has a solution if $c\le C_{N}$ and has 
no solution otherwise. Since $j_{\mu ,1}/\mu \to 1$ 
as $\mu \to \infty $, 
our condition \eqref{;cond1} is asymptotically optimal 
with respect to the dimension $N$, in the sense that 
as $N\to \infty $, 
\begin{align*}
 \frac{1}{2}j_{\frac{N-2}{2} ,1}^{2}
 \times \frac{1}{C_{N}}\to 1. 
\end{align*} 
The critical value $C_{N}$ also appears as the best constant of 
Hardy's inequality in $\R ^{N}$ 
as will be remarked in \sref{;simpr}. 
We derive the condition \eqref{;cond1} by adopting the same reasoning as in 
the proof of \tref{;t0} by Baras-Goldstein, with improvement and 
simplification of estimates given there. 
The following lemma is a key ingredient in the derivation: 
\begin{lem}\label{;keylem}
It holds that for all $T>0$,  
 \begin{align*}
  \int _{\{ \xi \in \R ^{N};\,|\xi |<1\} }
 P_{\xi }\left( 
 \max _{0\le s\le T}|B_{s}|<1
 \right) d\xi 
 \ge \frac{2\varpi _{N}}{j_{\frac{N-2}{2},1}^2}
 \exp \left( 
 -\frac{1}{2}j_{\frac{N-2}{2},1}^2T
 \right) , 
 \end{align*}
 where $\varpi _{N}=\frac{2\pi ^{N/2}}{\Gamma (N/2)}$ is the surface area of the $(N-1)$-dimensional unit sphere. This estimate is 
 also valid when $N=1,2$. 
\end{lem}
This lemma is proved by using eigenvalue expansions given in \cite{ken}  
for hitting distributions of Bessel processes. Note that the constant 
$\frac{1}{2}j_{\frac{N-2}{2},1}^2$ is equal to the smallest eigenvalue of minus 
one half the Dirichlet Laplacian in the unit ball in $\R ^{N}$. 

Another objective of the paper is, with replacing 
$(1/2)\Delta $ in the equation \eqref{;hq0} by the 
fractional Laplacian $-(-\Delta )^{\al /2}$ for 
$0<\al <2$, to give a sufficient condition on 
$V$ for the nonexistence of solutions to the equation. 
To be more precise, we replace in the expectation \eqref{;fkbr} the Brownian motion 
$(\{ B_{t}\} _{t\ge 0}, \{ P_{x}\} _{x\in \R ^{N}})$ by an $N$-dimensional 
rotationally invariant $\al $-stable process, where we allow the 
dimension 
$N$ to be less than $3$, and of concern is the transient case $N>\al $; we prove 
that the expectation diverges for any 
$x\in \R ^{N}$ and $t>0$ if 
\begin{align}\label{;cond2}
 \liminf _{r\to 0+}r^{\al }\vi (r)>j_{\frac{N-2}{2},1}^{\al }
\end{align}
and $V(x)\ge \vi (|x|)$ for a.e.\ $x\in \R ^{N}$. See \tref{;tmain2}.  
The proof is based on 
the representation of $\al $-stable process as a subordinated Brownian 
motion and \lref{;keylem} stated above. 
Similarly to the case of Brownian motion (i.e., the case $\al =2$), the 
constant $j_{\frac{N-2}{2},1}^{\al }$ in \eqref{;cond2} asymptotically 
coincides with the best constant of the Hardy-type inequality 
for the fractional Laplacian as will be seen in \sref{;sfrac}. 

Let $N\ge 3$ as in the case of Brownian motion. 
In \cite{ii}, Ishige and Ishiwata 
studied the existence and nonexistence of solutions to 
the heat equation in the half-space $\R ^{N}_{+}=\R ^{N-1}\times (0,\infty )$ with a 
singular potential on the boundary. 
In connection with their work, we are also concerned with expectations of the 
type 
\begin{align}\label{;fktr}
 \E _{x}\left[ 
 u_{0}(\bd _{t},|\bn _{t}|)\exp 
 \left\{ 
 \int _{0}^{t}V(\bd _{s},0)\,d\ln _{s}
 \right\} 
 \right] 
\end{align}
for $x=(x',x_{N})\in \R ^{N}_{+}$ and $t>0$, where 
under the probability measure $\P _{x}$, $\{ \bd _{t}\} _{t\ge 0}$ 
is an $(N-1)$-dimensional Brownian motion starting from $x'$, 
$\{ \bn _{t}\} _{t\ge 0}$ is a one-dimensional Brownian motion starting from 
$x_{N}$ and independent of $\bd $, 
and $\{ \ln _{t}\} _{t\ge 0}$ is the local time process of 
$\bn $ at the origin; $V$ is a measurable function on the boundary 
of $\R ^{N}_{+}$  
and we assume that $u_{0}$ is in $C_{0}(\R ^{N}_{+})$, nonnegative and not identically 
equal to 0. 
We show in \tref{;tmain3} that if 
\begin{align}\label{;cond3}
 \liminf _{r\to 0+}r\vi (r)>j_{\frac{N-3}{2},1}
\end{align}
and $V(x',0)\ge \vi (|x'|)$ for a.e.\ $x'\in \R ^{N-1}$, then the 
expectation \eqref{;fktr} diverges for any $x\in \R ^{N}_{+}$ and $t>0$. 
We also discuss a connection of the condition 
\eqref{;cond3} with the best constant of Kato's inequality in 
$\R ^{N}_{+}$. 

This paper is organized as follows: 
In \sref{;simpr}, we prove \tref{;tmain1} which asserts that \tref{;t0} holds true 
with the condition \eqref{;cond0} replaced by \eqref{;cond1}. 
In \sref{;sfrac}, we deal with the case of fractional Laplacians 
and see how the condition \eqref{;cond2} is derived in the proof of \tref{;tmain2}.  
\sref{;sbdry} concerns expectations of the form \eqref{;fktr}, which  
are seen in \tref{;tmain3} to be divergent if the condition \eqref{;cond3} 
is fulfilled. Those three \tsref{;tmain1}, \ref{;tmain2} and 
\ref{;tmain3} are proved in a unified manner 
by using \lref{;keylem}. The proof of \lref{;keylem} is given 
in the appendix, where we also discuss a connection of the expression \eqref{;fktr} 
with relativistic $1$-stable process in terms of the Laplace transform. 
\smallskip 

Throughout the paper, for every positive integer $d\in \N $ and every $t>0$, 
we denote 
by $g_{d}(t,\cdot )$ the Gaussian kernel on $\R ^{d}$: 
\begin{align*}
 g_{d}(t,x)
 :=\frac{1}{\sqrt{(2\pi t)^{d}}}
 \exp \left( 
 -\frac{|x|^{2}}{2t}
 \right) ,\quad x\in \R ^{d}. 
\end{align*}
For given two sequences $\{ a_{n}\} $, 
$\{ b_{n}\} $ of real numbers with $a_{n}\neq 0$ 
for all $n$, we write 
\begin{align*}
 a_{n}\sim b_{n} \quad \text{as }n\to \infty 
\end{align*}
to mean that $\lim \limits_{n\to \infty }b_{n}/a_{n}=1$. 
The symbol $\vi $ denotes a nonnegative 
measurable function on $(0,\infty )$ that is nonincreasing near the origin 
as mentioned above. 
Other notation will be introduced as needed.

\section{Improvement of the condition \eqref{;cond0}}\label{;simpr}

In this section we let $N\ge 3$ and $V$ a measurable function on $\R ^{N}$. 
The purpose of this section is to give a proof of 

\begin{thm}\label{;tmain1}
Suppose that $\vi $ satisfies \eqref{;cond1} and that 
$V(x)\ge \vi (|x|)$ for a.e.\ $x\in \R ^{N}$. Then 
the equation \eqref{;hq0} does not have a solution for any initial datum 
$u_{0}\in C_{0}(\R ^{N})$. 
\end{thm}

For each $m\in \N $, we set 
$V_{m}(x)=\min \{ m,V(x)\} ,\,x\in \R ^{N}$. Then the equation 
\eqref{;hq0} with $V$ replaced by $V_{m}$ has a unique solution 
$u_{m}$, and by the Feynman-Kac formula, it admits the representation 
\begin{align}\label{;fkbrn}
 u_{m}(t,x)=
 \ex _{x}
 \left[ 
 u_{0}(B_{t})\exp 
 \left( 
 \int _{0}^{t}V_{m}(B_{s})\,ds
 \right) 
 \right] , \quad t>0,\, x\in \R ^{N}. 
\end{align}
Here $\{ B_{t}\} _{t\ge 0}$ is an $N$-dimensional Brownian motion 
starting from $x$ under the probability measure $P_{x}$. Following 
Baras-Goldstein \cite{bg}, we say that the equation \eqref{;hq0} does not have a 
solution if 
\begin{align}\label{;diver}
 \lim _{m\to \infty }u_{m}(t,x)=\infty 
\end{align}
for all $t>0$ and $x\in \R ^{N}$. Note that by the representation \eqref{;fkbrn} 
and the monotone convergence theorem, \eqref{;diver} is 
restated as the divergence of the expectation \eqref{;fkbr}, 
to which we are going to give a proof from now on. 
Fix $t>0$ and $x\in \R ^{N}$ arbitrarily. 
Since we assume that $u_{0}$ is continuous and $u_{0}\ge (\not\equiv )\,0$, 
there exist $\ep _{0}>0$ and a nonempty open disc $D\subset \R ^{N}$ such that 
\begin{align}\label{;bdb}
 u_{0}(y)\ge \ep _{0}\quad \text{for all }y\in D. 
\end{align}
We fix $\A \in (0,1/2)$. 
Following the proof of \tref{;t0} by \cite{bg}, we set an event 
$A_{n}$ for each $n\in \N $ by 
\begin{align*}
 A_{n}=
 \left\{ 
 \max _{\A t\le s\le (1-\A )t}\left| B_{s}\right| <1/n, \ 
 B_{t}\in D
 \right\} . 
\end{align*}
We take $n_{0}\in \N $ so that $\vi $ is nonincreasing on $(0,1/n_{0}]$. 
Then for $n\ge n_{0}$, by restricting the $\pr _{x}$-expectation 
in \eqref{;fkbr} to $A_{n}$ and using \eqref{;bdb}, we see that 
\eqref{;fkbr} is bounded from below by 
\begin{align}
 &\ep _{0}\ex _{x}\left[ 
 \exp \left\{ 
 \int _{\A t}^{(1-\A )t}V(B_{s})\,ds
 \right\} ;\,A_{n}
 \right] \notag \\
 &\ge \ep _{0}\exp \left\{ 
 \vi \left( \frac{1}{n}\right) \ga t
 \right\} \pr _{x}(A_{n}), \label{;1e1}
\end{align}
where we set $\ga =1-2\A $. For $\pr _{x}(A_{n})$, we have the 
following estimate: set $\mu =(N-2)/2$. 

\begin{prop}\label{;pI1}
 There exists a positive constant 
 $\Cst \equiv \Cst (x,t,\A ,D,N)$ independent of $n$ such that 
 \begin{align*}
  \pr _{x}(A_{n})\ge \Cst \left( \frac{1}{n}\right) ^{N}
  \exp\left( 
  -\frac{1}{2}j_{\mu ,1}^2n^2\ga t
  \right) \quad \text{for all }n\in \N . 
 \end{align*}
 This estimate also holds true in the case $N=1,2$. 
\end{prop}

Once this proposition is shown, the proof of \tref{;tmain1} is 
immediate: 

\begin{proof}[Proof of \tref{;tmain1}]
By \eqref{;1e1} and \pref{;pI1}, the expectation \eqref{;fkbr} 
is bounded from below by 
\begin{align*}
 \ep _{0}\Cst \left( \frac{1}{n}\right) ^{N}
 \exp \left\{ 
 \left( 
 \vi \left( \frac{1}{n}\right) 
 -\frac{1}{2}j_{\mu ,1}^{2}n^2
 \right) \ga t
 \right\} , 
\end{align*}
which tends to infinity as $n\to \infty $ under the condition 
\eqref{;cond1}. Therefore the assertion is proved. 
\end{proof}

It remains to prove \pref{;pI1}. 
\begin{proof}[Proof of \pref{;pI1}]
By the Markov property of Brownian motion, we have 
\begin{align*}
 \pr _{x}(A_{n})=\ex_{x}\left[ 
 \vp \left( B_{\A t}\right) ;\,|B_{\A t}|<1/n
 \right] , 
\end{align*}
where we set 
\begin{align*}
 \vp (y)=\pr _{y}\left( 
 \max _{0\le s\le \gamma t}|B_{s}|<1/n,\ B_{(1-\A )t}\in D
 \right) , \quad y\in \R ^{N}. 
\end{align*}
Using the Markov property again, we further have for all $y\in \R ^{N}$, 
\begin{align*}
 \vp (y)&=\ex _{y}\left[ 
 P_{B_{\gamma t}}
 \left( 
 B_{\A t}\in D\right) ;\,\max \limits_{0\le s\le \gamma t}|B_{s}|<1/n
 \right] \\
 &\ge \inf _{|z|\le 1/n}P_{z}\left( B_{\A t}\in D\right) 
 \times P_{y}\left( 
 \max _{0\le s\le \gamma t}|B_{s}|<1/n
 \right) \\
 &\ge c_{1}P_{y}\left( 
 \max _{0\le s\le \gamma t}|B_{s}|<1/n
 \right) , 
\end{align*}
where 
$
 c_{1}:=\inf \limits_{|z|\le 1}
 P_{z}\left( 
 B_{\A t}\in D
 \right) 
$, 
which is positive since 
$\R ^{N}\ni z\mapsto P_{z}\left( B_{\A t}\in D\right) $ is 
continuous. Therefore we have the estimate 
\begin{align*}
 \pr _{x}(A_{n})&\ge c_{1}\ex _{x}\left[ 
 P_{B_{\A t}}
 \left( 
 \max _{0\le s\le \gamma t}|B_{s}|<1/n
 \right) ;\,|B_{\A t}|<1/n
 \right] \notag \\
 &=c_{1}\int _{|y|<1/n}
 dy\,g_{N}(\A t,y-x)
 P_{y}\left( 
 \max _{0\le s\le \gamma t}|B_{s}|<1/n
 \right) \notag \\
 &=c_{1}\left( \frac{1}{n}\right) ^{N}
 \int _{|\xi |<1}d\xi \,g_{N}(\A t,\xi /n-x)
 P_{\xi /n}\left( 
 \max _{0\le s\le \gamma t}|B_{s}|<1/n
 \right) \notag \\
 &\ge c_{1}c_{2}\left( \frac{1}{n}\right) ^{N}
 \int _{|\xi |<1}d\xi \,
 P_{\xi }\left( 
 \max _{0\le s\le n^2\gamma t}|B_{s}|<1
 \right) 
\end{align*}
with 
$
 c_{2}:=
 \inf \limits_{|\xi |\le 1}g_{N}(\A t,\xi -x)
>0
$ 
in the last line, where we also used the scaling property of 
Brownian motion.  
The proposition follows by taking $T=n^2\ga t$ in 
\lref{;keylem}. 
\end{proof}

We end this section with a remark on \tref{;tmain1}. 

\begin{rem}
\thetag{1}~For every real $\delta \ge 2$ and 
$r>0$, we denote by 
$\bigl( \{ R_{t}\} _{t\ge 0},P^{(\delta )}_{r}\bigr) $ 
a $\delta $-dimensional Bessel process starting 
from $r$. It is known \cite{yor} that Bessel processes 
enjoy the following absolute continuity relationship: 
for every $t>0$ and every nonnegative measurable 
functional $F$ on the space $C([0,t];\R )$ of real-valued 
continuous paths over $[0,t]$,  
\begin{align*}
  E^{(\delta )}_{r}\left[ 
  F(R_{s},s\le t)
  \right] 
  =E^{(2)}_{r}\left[ 
  F(R_{s},s\le t)\left( \frac{R_{t}}{r}\right) ^{\mu }
  \exp \left( -\frac{1}{2}\,\mu ^2\int _{0}^{t}
  \frac{ds}{R_{s}^2}\right) 
  \right] , 
 \end{align*}
where $\mu =\delta /2-1$. 
Take $\delta =N$ with $N\ge 3$. In the expression \eqref{;fkbr}, 
suppose that $u_{0}$ is rotationally invariant, namely 
$u_{0}(x)=f(|x|)$ for all $x\in \R ^{N}$ for some 
nonnegative function $f$ on $(0,\infty )$, and that $V$ is of the form 
$V(x)=c/|x|^2$ with $c$ a positive constant. Then by the above 
relationship, \eqref{;fkbr} is written as 
\begin{align}\label{;fkbes}
 &\ex _{x}\left[ 
 f\left( |B_{t}|\right)
 \exp 
 \left( 
 c\int _{0}^{t}
 \frac{ds}{|B_{s}|^2}
  \right)  
 \right] \notag \\
 &=E^{(2)}_{|x|}
 \left[ 
 f(R_{t})
 \left( 
 \frac{R_{t}}{|x|}
 \right) ^{\frac{N}{2}-1}
 \exp \left\{ 
 (c-C_{N})\int _{0}^{t}\frac{ds}{R_{s}^{2}}
 \right\} 
 \right] 
\end{align}
when $x\neq 0$. Here 
$C_{N}=\frac{1}{2}\left( \frac{N-2}{2}\right) ^{2}$ as introduced 
in \sref{;intro}. 
It is clear that if $c\le C_{N}$ and $f$ is compactly supported, then 
\eqref{;fkbes} is finite; moreover, by the fact that 
\begin{align}\label{;divtrans}
 E^{(2)}_{|x|}\left[ 
 \frac{1}{R_{s}^{2}}\,\bigg| \,R_{t}=y
 \right] =\infty \quad \text{for a.e.\ }y>0 
\end{align}
for any $0<s<t$, the 
expectation \eqref{;fkbes} is divergent as long as 
$\left| \{ f>0\} \right| >0$ in the case $c>C_{N}$.  
This observation agrees with \cite[Theorem~2.2]{bg}. 
The fact \eqref{;divtrans} is easily deduced from 
the explicit representation for the transition density functions of 
Bessel process (see, e.g., \cite[Chapter~XI]{rey}). 
See also \rref{;rtmain2}\,\thetag{2} in the next section. 

\noindent 
\thetag{2}~
Also explicitly known is the following joint distribution 
\cite[p.\,386, Formula~1.20.8]{bs}: 
\begin{align}\label{;expljoint}
 &P^{(\delta )}_{r}\left( 
 \int _{0}^{t}\frac{ds}{R_{s}^2}\in dz,\,R_{t}\in d\xi 
 \right) \notag \\
 &=\frac{1}{t}\left( \frac{\xi }{r}\right) ^{\mu }\xi 
 \exp \left( 
 -\frac{1}{2}\mu ^2z-\frac{r^2+\xi ^2}{2t}
 \right) \theta _{r\xi /t}(z)\,dzd\xi ,\quad z,\xi >0, &
\end{align}
for any $r>0$ and $t>0$, where for every $\rho >0$, $\theta _{\rho }$ 
is a constant multiple of the density function of the Hartman-Watson 
distribution on $(0,\infty )$, whose integral representation is given in \cite{yor}: 
\begin{align*}
 &\theta _{\rho }(z)=\frac{\rho }{\sqrt{2\pi ^3z}}
 \int _{0}^\infty dy\,\exp \left( 
 \frac{\pi ^2-y^2}{2z}
 \right) \exp \left( 
 -\rho \cosh y
 \right) \sinh y\sin \left( \frac{\pi y}{z}\right) ,\quad z>0. 
\end{align*}
By this expression, we have in particular 
\begin{align*}
 \lim _{z\to \infty }\sqrt{2\pi z^{3}}
 \theta _{\rho }(z)&=\rho \int _{0}^{\infty }dy\,y\exp \left( 
 -\rho \cosh y
 \right) \sinh y\\
 &=\int _{0}^{\infty }dy\,\exp \left( 
 -\rho \cosh y
 \right) \\
 &=K_{0}(\rho ), 
\end{align*}
where $K_{0}$ is the modified Bessel function of the third kind 
(Macdonald function) with index $0$. From this asymptotics 
and \eqref{;expljoint}, we see that 
for every $x\in \R ^{N}$\,($x\neq 0$) and $t>0$, 
\begin{align}\label{;convdiv1}
  \ex _{x}\left[ 
  \exp \left( c\int _{0}^{t}\frac{ds}{|B_{s}|^{2}}\right) 
  \right] 
  \begin{cases}
   <\infty & \text{if }c\le C_{N}, \\
   =\infty & \text{if }c>C_{N}, 
  \end{cases}
\end{align}
which is consistent with the observation in \thetag{1}. 
We remark that since by the scaling property, 
\begin{align*}
 \ex _{0}\left[ 
 \int _{0}^{t}\frac{ds}{|B_{s}|^{2}}
 \right] &=\int _{0}^{t}\frac{ds}{s}\times 
 \ex _{0}\left[ 
 \frac{1}{|B_{1}|^{2}}
 \right] \\
 &=\infty  
\end{align*}
for any $t>0$, we cannot draw a sufficient condition on $c$ 
for the finiteness of expectations in \eqref{;convdiv1} from 
Khas'minskii's well-known lemma (see, e.g., \cite[Lemma~3.7]{cz}). 

\noindent 
\thetag{3}~The constant $C_{N}$ coincides with the best constant 
of Hardy's inequality: 
\begin{align*}
 C_{N}\int _{\R ^{N}}
 \frac{|\phi (x)|^2}{|x|^{2}}\,dx
 \le \int _{\R ^{N}}\phi (x)
 \left( -\frac{1}{2}\Delta \phi (x)\right) dx, 
 \quad \phi \in C_{0}^{\infty }(\R ^{N}).  
\end{align*}
The factor $1/2$ in the right-hand side is put in accordance with 
\eqref{;hq0}. \tref{;tmain1} indicates that 
$
\frac{1}{2}j_{\frac{N-2}{2},1}^{2}\ge C_{N}
$; in fact, the following upper and lower estimates are known 
\cite{cham,lor} as to $j_{\mu ,1}$ for $\mu >-1$: 
\begin{align}\label{;j1asym}
 \sqrt{(\mu +1)(\mu +5)}\le 
 j_{\mu ,1}\le \sqrt{\mu +1}\left( \sqrt{\mu +2}+1\right). 
\end{align}
For more precise bounds, see, e.g., \cite{qw} 
(see also \cite[Chapter~5]{leb} for detailed descriptions of 
Bessel functions). These estimates reveal that 
the constant $\frac{1}{2}j_{\frac{N-2}{2},1}^{2}$ is asymptotically optimal 
in the sense that 
\begin{align*}
 \frac{1}{2}j_{\frac{N-2}{2},1}^{2}\sim C_{N} 
 \quad \text{as }N\to \infty . 
\end{align*}
\end{rem}

\section{The case of fractional Laplacians}\label{;sfrac}

In this section the dimension $N$ is allowed to be less than $3$.  
Fix $0<\al <2$. For each $x\in \R ^{N}$, we denote by 
$(\{ X_{t}\} _{t\ge 0},\pr _{x})$ an $N$-dimensional rotationally invariant 
$\al $-stable process starting from $x$, that is, under the 
probability measure $\pr _{x}$, the process $X_{t}-x$,\,$t\ge 0$, 
is a L\'evy process whose characteristic function is given by 
\begin{align*}
 \ex _{x}\left[ 
 \exp \left\{ 
 i\xi \cdot (X_{t}-x)
 \right\} 
 \right] =e^{-t|\xi |^{\al }}, 
 \quad t\ge 0,\,\xi \in \R ^{N}; 
\end{align*}
recall that the process $(\{ X_{t}\} _{t\ge 0},\{ \pr _{x}\} _{x\in \R ^{N}})$ 
is a right-continuous Markov process with infinitesimal generator 
$-(-\Delta )^{\al /2}$. Throughout the section, unless otherwise stated, we 
assume $N>\al $, i.e., 
we deal with the transient case (see \rref{;rtmain2}\,\thetag{2} as to this 
condition on $N$).  
The same as in the previous section, we let $V$ be a measurable function 
on $\R ^{N}$ and assume that $u_{0}\in C_{0}(\R ^{N})$ is nonnegative and 
not identically equal to $0$. The purpose of this section is to prove 

\begin{thm}\label{;tmain2}
Suppose that $\vi $ 
satisfies the condition \eqref{;cond2} and that 
$V(x)\ge \vi (|x|)$ for a.e.\ $x\in \R ^{N}$. Then 
\begin{align}\label{;divfr}
 \ex _{x}\left[ 
 u_{0}(X_{t})\exp 
 \left( 
 \int _{0}^{t}V(X_{s})\,ds
 \right) 
 \right] =\infty 
\end{align}
for any $x\in \R ^{N}$ and $t>0$. 
\end{thm}

To prove the theorem, we first recall that the $\al $-stable process 
$X$ is identical in law with a subordinated Brownian motion. Let 
$\{ \sub{t}\} _{t\ge 0}$ be an $\al /2$-stable subordinator under 
a probability measure $P$, that is, $\sub{}$ is a nondecreasing 
L\'evy process characterized by 
\begin{align}\label{;subord}
 E\left[ 
 e^{-\la \sub{t}}
 \right] =e^{-t\la ^{\al /2}} \quad \text{for all }
 \la ,t\ge 0. 
\end{align}
Let $\{ W(t)\} _{t\ge 0}$ be an $N$-dimensional standard Brownian motion 
under $P$, independent of $\sub{}$. Then it is known that 
the following identity in law holds: 
\begin{align}\label{;idenlaw}
 \left( \{ X_{t}\} _{t\ge 0}, \pr _{x}\right) 
 \stackrel{(d)}{=}
 \left( \left\{ 
 x+W(2\sub{t})
 \right\} _{t\ge 0},P\right) ; 
\end{align}
for subordinators and stable processes, see \cite[Chapter~1]{ap}. 
Using this identity and \lref{;keylem}, we prove \tref{;tmain2}. 
As in the previous section, we fix $\A \in (0,1/2)$ and set 
$\ga =1-2\A $; we also let a positive $\ep _{0}$ and a nonempty open disc 
$D\subset \R ^{N}$ be such that $u_{0}$ fulfills \eqref{;bdb}. 

\begin{proof}[Proof of \tref{;tmain2}]
 For each $n\in \N $, set 
 \begin{align*}
  A_{n}=\left\{ 
  \max _{\A t\le s\le (1-\A )t}|X_{s}|<1/n,\,
  X_{t}\in D
  \right\} . 
 \end{align*}
 Then by arguing in the same way as in the proof of \tref{;tmain1}, 
 the left-hand side of \eqref{;divfr} 
 is bounded from below by 
 \begin{align}\label{;2e1}
  \ep _{0}\exp \left\{ 
  \vi \left( 
  \frac{1}{n}
  \right) \ga t
  \right\} \pr _{x}(A_{n}) 
 \end{align}
 for every sufficiently large $n$. 
 By the Markov property of $\al $-stable process, 
 \begin{align}
  \pr _{x}(A_{n})&=\ex _{x}\left[ 
  \pr _{X_{(1-\A )t}}\left( X_{\A t}\in D\right) ;\,
  \max _{\A t\le s\le (1-\A )t}|X_{s}|<1/n
  \right] \notag \\
  &\ge c_{1}\pr _{x}\left( 
  \max _{\A t\le s\le (1-\A )t}|X_{s}|<1/n
  \right), \label{;2e2}
 \end{align}
 where $c_{1}:=\inf \limits_{|z|\le 1}\pr _{z}\left( 
 X_{\A t}\in D
 \right) $, which is positive since by \eqref{;idenlaw}, 
 \begin{align*}
  c_{1}&=\inf _{|z|\le 1}\int _{0}^{\infty }
  P(\sub{\A t}\in ds)
  P\left( 
  z+W(2s)\in D
  \right) \\
  &\ge c_{1}'\times P(1\le \sub{\A t}\le 2)\times 
  |D|
 \end{align*}
 with 
 \begin{align*}
  c_{1}':=
  \inf \left\{ 
  g_{N}(2s,y-z)
  ;\,
  1\le s\le 2,\,y\in \overline{D},\,|z|\le 1
  \right\} >0. 
 \end{align*}
 By  the Markov property and \eqref{;idenlaw}, 
 the probability in the right-hand side of \eqref{;2e2} 
 is written as 
 \begin{align*}
  &\ex _{x}\left[ 
  \pr _{X_{\A t}}\left( 
  \max _{0\le s\le \ga t}|X_{s}|<1/n
  \right) ;\,|X_{\A t}|<1/n
  \right] \\
  &=\int _{0}^{\infty }
  \pr (\sub{\A t}\in ds)
  \int _{|y|<1/n}
  dy\,g_{N}(2s,y-x)
  \pr _{y}\left( 
  \max _{0\le s\le \ga t}|X_{s}|<1/n
  \right) . 
 \end{align*}
 Therefore setting a positive constant $c_{2}$ by 
 \begin{align*}
  c_{2}=
  P(1\le \sub{\A t}\le 2)\times 
  \inf \left\{ 
   g_{N}(2s,y-x);\,1\le s\le 2,\,|y|\le 1
  \right\} , 
 \end{align*}
 we see from \eqref{;2e2} that 
 \begin{align}\label{;2e3}
  \pr _{x}(A_{n})\ge c_{1}c_{2}
  \int _{|y|<1/n}dy\,
  \pr _{y}\left( 
  \max _{0\le s\le \ga t}|X_{s}|<1/n
  \right) . 
 \end{align}
 By \eqref{;idenlaw}, the integrand in the right-hand side of \eqref{;2e3} 
 is rewritten and estimated as 
 \begin{align*}
  &P \left( 
  \max _{0\le s\le \ga t}\left| 
  y+W(2\sub{s})
  \right| <1/n
  \right) \\
  &\ge P\left( 
  \max _{0\le s\le 2\sub{\ga t}}\left| 
  y+W(s)
  \right| <1/n
  \right) \\
  &=\int _{0}^{\infty }
  P(\sub{\ga t}\in d\tau )
  P\left( 
  \max _{0\le s\le 2\tau }|y+W(s)|<1/n
  \right) , 
 \end{align*}
 where the inequality is due to the fact that $\sub{}$ may have a jump. 
 Plugging this estimate into \eqref{;2e3}, we have by Fubini's theorem 
 and the scaling property of Brownian motion, 
 \begin{align}
  \pr _{x}(A_{n})&\ge 
  c_{1}c_{2}\left( \frac{1}{n}\right) ^{N}
  \int _{0}^{\infty }P(\sub{\ga t}\in d\tau )
  \int _{|\xi |<1}d\xi\,
  P\left( 
  \max _{0\le s\le 2n^2\tau }|\xi +W(s)|<1
  \right) \notag \\
  &\ge c_{1}c_{2}\left( \frac{1}{n}\right) ^{N}
  \times \frac{2\varpi _{N}}{j_{\frac{N-2}{2},1}^{2}}
  \int _{0}^{\infty }P(\sub{\ga t}\in d\tau )
  \exp \left( 
  -j_{\frac{N-2}{2},1}^{2}n^2\tau 
  \right) \notag \\
  &=\frac{2\varpi _{N}}{j_{\frac{N-2}{2},1}^{2}}c_{1}c_{2}
  \left( 
  \frac{1}{n}
  \right) ^{N}
  \exp \left( 
  -j_{\frac{N-2}{2},1}^{\al }n^{\al }\ga t
  \right) , \label{;2e4}
 \end{align}
 where we used \lref{;keylem} with $T=2n^2\tau $ 
 for the second line and \eqref{;subord} for the third. 
 By \eqref{;2e4}, we see that \eqref{;2e1} diverges 
 as $n\to \infty $ under the condition \eqref{;cond2}, 
 which ends the proof. 
\end{proof}

We conclude this section with a remark on \tref{;tmain2}. 
\begin{rem}\label{;rtmain2}
\thetag{1}~We recall the Hardy-type inequality for the fractional 
Laplacian $-(-\Delta )^{\al /2}$ in $\R ^{N}$ with $N>\al $: 
\begin{align*}
 C_{N,\al }\int _{\R ^{N}}
 \frac{|\phi (x)|^{2}}{|x|^{2}}\,dx
 \le \int _{\R ^{N}}\phi (x)\left( 
 (-\Delta )^{\al /2}\phi (x)
 \right) dx, \quad \phi \in C_{0}^{\infty }(\R ^{N}), 
\end{align*}
where 
\begin{align}\label{;opfrac}
 C_{N,\al }:=2^{\al }
 \frac{\Gamma ^{2}\left( \frac{N+\al }{4}\right) }
 {\Gamma ^{2}\left( \frac{N-\al }{4}\right) }
\end{align}
with $\Gamma $ denoting the gamma function, is the best constant; see, e.g., \cite{her,fls}. 
The constant $j_{\frac{N-2}{2},1}^{\al }$ in the condition 
\eqref{;cond2} asymptotically recovers this optimal $C_{N,\al }$: 
\begin{align*}
 j_{\frac{N-2}{2},1}^{\al }
 \sim C_{N,\al } \quad \text{as }N\to \infty . 
\end{align*}
Indeed, the estimates \eqref{;j1asym} on $j_{\mu ,1}$ shows the asymptotics 
\begin{align*}
 j_{\frac{N-2}{2},1}^{\al }\sim \left( \frac{N}{2}\right) ^{\al }, 
\end{align*}
which $C_{N,\al }$ admits as well by Stirling's formula. 
In view of \eqref{;convdiv1}, it is plausible that 
for every $x\in \R ^{N}$\,($x\neq 0$) and $t>0$, 
\begin{align*}
  \ex _{x}\left[ 
  \exp \left( c\int _{0}^{t}\frac{ds}{|X_{s}|^{\al }}\right) 
  \right] 
  \begin{cases}
   <\infty & \text{if }c\le C_{N,\al }, \\
   =\infty & \text{if }c>C_{N,\al }. 
  \end{cases}
\end{align*}

\noindent 
\thetag{2}~In the case $N\le \al $ it holds that for any $\ep >0$, 
\begin{align}\label{;condfr}
 \ex _{x}
 \left[ 
 \frac{1}{|X_{s}|^{\al }}\ind _{\{ |X_{s}|<\ep \} }\Big| \,
 X_{t}=y
 \right] =\infty \quad \text{for a.e.\ }y\in \R ^{N}
\end{align}
for every $0<s<t$. Indeed, by denoting the transition density function of 
$X$ by 
$p^{\al }_{t}(x,y),\,t>0,x,y\in \R ^{N}$, the left-hand side of \eqref{;condfr}  is 
written, for a.e.\ $y$, as 
\begin{align*}
 \int _{|z|<\ep }\frac{dz}{|z|^{\al }}
 \frac{p^{\al }_{s}(x,z)p^{\al }_{t-s}(z,y)}{p^{\al }_{t}(x,y)}, 
\end{align*}
which is rewritten, by changing to polar coordinates, as 
\begin{align*}
 \int _{(0,\ep )}dr\,r^{N-\al -1}
 \int _{\mathbb{S}^{N-1}}\sigma (dw)\,
 \frac{p^{\al }_{s}(x,rw)p^{\al }_{t-s}(rw,y)}{p^{\al }_{t}(x,y)}
\end{align*}
with $\mathbb{S}^{N-1}$ and $\sigma $ being the $(N-1)$-dimensional 
unit sphere and the surface element on $\mathbb{S}^{N-1}$, respectively. 
By this expression, 
we have \eqref{;condfr} if $N-\al -1\le -1$, i.e., $N\le \al $. 
\end{rem}

\section{Heat equation with a singular potential on the boundary}\label{;sbdry}

In this section we let $N\ge 3$. 
We denote by 
$(\{ \B _t\} _{t\ge 0}, \{ \P _{x}\} _{x\in \R ^N})$ an 
$N$-dimensional Brownian motion and 
by $\E _{x}$ the expectation relative to the probability measure 
$\P _{x}$. Set $\R ^{N}_{+}=\R ^{N-1}\times (0,\infty )$. For 
$x=(x',x_{N})\in \R ^{N}_{+}$, we write 
$\B _{t}=(\bd _t,\bn _t),\,t\ge 0$, where under $\P _{x}$, 
$\bd $ is the $(N-1)$-dimensional Brownian motion starting from 
$x'\in \R ^{N-1}$ that consists of the first $(N-1)$ coordinates of $\B $, and 
$\bn $ is the one-dimensional Brownian motion starting from 
$x_{N}>0$, given as the $N$th coordinate of $\B $. Note that two 
processes 
$\bd $ and $\bn $ are independent. We denote by 
$\{ \ln _{t}\} _{t\ge 0}$ the local time process of $\bn $ 
at the origin, which is given through Tanaka's formula: 
\begin{align}\label{;it}
\left| 
\bn _{t}
\right| 
=x_{N}+\int _{0}^{t}\sgn \bn _{s}\,d\bn _{s}+\ln _{t}, \quad 
t\ge 0\quad \P _{x}\text{-a.s., }
\end{align}
where $\mathrm{sgn}\,a$ denotes the signature of $a\in \R $. 
Let $V$ be a measurable function on 
$\partial \R ^{N}_{+}=\R ^{N-1}\times \{ 0\} $. 
The purpose of this section is to prove the following theorem: 
\begin{thm}\label{;tmain3}
 Let $u_{0}\in C_{0}(\R ^{N}_{+})$ be nonnegative and not identically 
 equal to $0$. 
 Suppose that $\vi $ satisfies the condition \eqref{;cond3} and that 
 $V(x',0)\ge \vi (|x'|)$ for a.e.\ $x'\in \R ^{N-1}$. Then 
 the expectation \eqref{;fktr} diverges for any $x\in \R ^{N}_{+}$ and $t>0$. 
\end{thm}

\subsection{Feynman-Kac formula for a boundary value problem}
\label{;ss41} 

Before giving a proof of \tref{;tmain3}, we explain where expectations of the 
form \eqref{;fktr} arise from. We consider the following initial-boundary value problem 
for the heat equation in $\R ^{N}_{+}$: 
\begin{align}\label{;hq1}
\left\{
\begin{array}{ll}
\displaystyle{\frac{\partial}{\partial t}} u-\frac{1}{2}\Delta u=0 & 
\mbox{in}\quad (0,\infty )\times \R ^{N}_{+},\vspace{7pt}\\
\displaystyle{\frac{\partial}{\partial x_N}} u+Vu=0 & 
\mbox{on}\quad (0,\infty )\times \partial \R ^{N}_{+},\vspace{7pt}\\
u(0,x)=u_0(x)
& \mbox{in}\quad \R ^{N}_{+}. 
\end{array}
\right. 
\end{align}
In what follows we often write $u(t,x)=u(t,x',x_{N})$ for $x=(x',x_{N})\in \R ^{N}_{+}$. 

\begin{prop}\label{;pFK}
Assume that $V$ is bounded and that the continuous function 
$u:[0,\infty )\times \R _{+}^{N}\to [0,\infty )$ is of class 
$C^{1,2}$ on $(0,\infty )\times \R _{+}^{N}$ and satisfies 
\eqref{;hq1}. Moreover, assume that for each finite $T>0$, 
there exist constants $K>0$ and $0<\lambda <1/(2NT)$ such that 
\begin{align}\label{;aFK}
 \max _{0\le t\le T}
 u(t,x)
 \le 
 Ke^{\lambda |x|^2} \quad \text{for all }x\in \R _{+}^{N}. 
\end{align}
Then for every $t\ge 0$ and $x\in \R _{+}^{N}$, $u(t,x)$ admits the representation 
\eqref{;fktr} . 
\end{prop}

\begin{proof}
Let $T>0$ be fixed and set 
 \begin{align*}
  M_{t}:=e^{A_{t}}u(T-t,\bd _{t},|\bn _{t}|), \quad 0\le t\le T, 
 \end{align*}
 where 
 \begin{align*}
  A_{t}:=\int _{0}^{t}V(\bd _{s},0)\,d\ln _{s}. 
 \end{align*}
 By It\^o's formula, it holds that $\P _{x}$-a.s., 
 \begin{align*}
  M_{t}=u(T,x)-&\int _{0}^{t}e^{A_{s}}\frac{\partial u}{\partial t}
  (T-s,\bd _{s},|\bn _{s}|)\,ds+\int _{0}^{t}e^{A_{s}}
  u(T-s,\bd _{s},|\bn _{s}|)\,dA_{s}\\
  +&\int _{0}^{t}e^{A_{s}}\nabla _{x'}u(T-s,\bd _{s},|\bn _{s}|)
  \cdot d\bd _{s}
  +\int _{0}^{t}e^{A_{s}}\frac{\partial u}{\partial x_{N}}
  (T-s,\bd _{s},|\bn _{s}|)\,d|\bn _{s}|\\
  +&\frac{1}{2}\int _{0}^{t}e^{A_{s}}\Delta u(T-s,\bd _{s},|\bn _{s}|)\,ds
 \end{align*}
for all $0\le t\le T$. As $u$ solves \eqref{;hq1}, the second and sixth 
terms on the right-hand side are cancelled. Moreover, by Tanaka's formula 
\eqref{;it} and by the boundary condition in \eqref{;hq1}, the sum of 
the third and fifth terms is equal to 
\begin{align*}
&\int _{0}^{t}e^{A_{s}}u(T-s,\bd _{s},0)V(\bd _{s},0)\,d\ln _{s}\\
&+\int _{0}^{t}e^{A_{s}}\frac{\partial u}{\partial x_{N}}
(T-s,\bd _{s},|\bn _{s}|)\sgn \bn _{s}\,d\bn _{s}\\
&+\int _{0}^{t}e^{A_{s}}\frac{\partial u}{\partial x_{N}}
(T-s,\bd _{s},0)\,d\ln _{s}\\
&=\int _{0}^{t}e^{A_{s}}\frac{\partial u}{\partial x_{N}}
(T-s,\bd _{s},|\bn _{s}|)\sgn \bn _{s}\,d\bn _{s}. 
\end{align*}
Here we used the fact that $d\ln _{s}$ is carried by the set 
$\{ s\ge 0; \bn _{s}=0\} $. Therefore we have $\P _{x}$-a.s., 
\begin{align*}
 M_{t}=u(T,x)+&\int _{0}^{t}e^{A_{s}}\nabla _{x'}u(T-s,\bd _{s},|\bn _{s}|)
 \cdot d\bd _{s}\\
 +&\int _{0}^{t}e^{A_{s}}\frac{\partial u}{\partial x_{N}}
 (T-s,\bd _{s},|\bn _{s}|)\sgn \bn _{s}\,d\bn _{s} 
\end{align*}
for all $0\le t\le T$. 
We follow the notation in the proof of \cite[Theorem~4.4.2]{ks} to define 
$S_{n}:=\inf \{ t>0;|\B _{t}|\ge n\sqrt{N}\} ,\,n\in \N $. 
By the continuity of $\nabla _{x'}u$ and 
$\frac{\partial u}{\partial x_{N}}$, and by the boundedness of $V$, 
we deduce that 
\begin{align*}
 \E _{x}\left[ 
 M_{T\wedge S_{n}}
 \right] 
 =u(T,x)
\end{align*}
for every $n\in \N $. In fact, as $\{ \ln _{t}\} _{t\ge 0}$ satisfies 
\begin{align}\label{;expm}
 \E _{x}\left[ 
 e^{\kappa \ln _{t}}
 \right] <\infty 
\end{align}
for all $\kappa >0$ and $t\ge 0$ (see \eqref{;jloc} below), 
the process $\{ M_{t\wedge S_{n}}\} _{0\le t\le T}$ 
is a square-integrable martingale, from which we have 
$\E _{x}[M_{T\wedge S_{n}}]=\E _{x}[M_{0}]=u(T,x)$. 
Since 
\begin{align*}
 M_{T}=e^{A_{T}}u_{0}(\bd _{T},|\bn _{T}|)
\end{align*}
by definition, it remains to prove 
\begin{align}\label{;qFK2}
 \lim _{n\to \infty }\E _{x}\left[ 
 M_{T\wedge S_{n}}
 \right] 
 =\E _{x}\left[ 
 M_{T}
 \right] . 
\end{align}
To this end, we divide $\E _{x}\left[ 
 M_{T\wedge S_{n}}
 \right] $ 
into the sum 
\begin{align*}
 \E _{x}\left[ 
 M_{T}\ind _{\{ S_{n}>T\} }
 \right] +
 \E _{x}\left[ 
 M_{S_{n}}\ind _{\{ S_{n}\le T\} }
 \right] . 
\end{align*}
Due to the nonnegativity of $u_{0}$, the first term converges to 
$\E _{x}\left[ 
 M_{T}
 \right] $ as $n\to \infty $ by the monotone convergence theorem. To see 
 that the second term converges to $0$, we fix an exponent $p>1$ so that 
 $\lambda p<1/(2NT)$ for $\lambda $ given in the condition \eqref{;aFK}, and use 
 the H\"older inequality to obtain 
\begin{align*}
 \E _{x}\left[ 
 M_{S_{n}}\ind _{\{ S_{n}\le T\} }
 \right] &=\E _{x}\left[ 
 e^{A_{S_{n}}}u\left( T-S_{n},\bd _{S_{n}},|\bn _{S_{n}}|\right) 
 \ind _{\{ S_{n}\le T\} }\right] \\
 &\le 
 \left\{ 
 \E _{x}\left[ 
 e^{qA_{S_{n}}}
 \ind _{\{ S_{n}\le T\} }\right] 
 \right\} ^{1/q}
 \times 
 \left\{ 
 Ke^{\lambda pNn^2}\P _{x}(S_{n}\le T)
 \right\} ^{1/p}, 
\end{align*}
where $q$ is the conjugate of $p$. Note that the first factor of the 
last member is bounded because of \eqref{;expm} and the boundedness of 
$V$. The second factor converges to $0$ as $n\to \infty $ by the same 
argument as in the proof of \cite[Theorem~4.4.2]{ks} since $\lambda p<1/(2NT)$. 
Therefore \eqref{;qFK2} is proved, which ends the proof of the 
proposition. 
\end{proof}

\begin{rem}
For the solvability of \eqref{;hq1} and a priori estimates on the 
unique solution, 
see \cite[Chapter~IV]{lsu}. 
\end{rem}

In \cite{ii}, Ishige and Ishiwata studied the problem \eqref{;hq1} 
in the case of a singular potential given by $V(x)=c/|x|,\,c>0$; 
employing a PDE approach, they showed the existence of 
the threshold number $C^{*}_{N}$ such that for any nonnegative 
initial datum $u_{0}\,(\not\equiv 0)$ in $C_{0}(\R ^{N}_{+})$, 
the equation \eqref{;hq1} has a solution 
if $c\le C^{*}_{N}$ and has no solution otherwise. The constant 
$C^{*}_{N}$ is characterized as the best constant of Kato's inequality in 
$\R ^{N}_{+}$: 
\begin{align*}
 C^{*}_{N}\int _{\partial \R ^{N}_{+}}
 \frac{|\phi (x)|^2}{|x|}\sigma (dx)\le 
 \int _{\R ^{N}_{+}}|\nabla \phi (x)|^{2}\,dx, 
 \quad \phi \in C_{0}^{\infty }(\R ^{N}_{+}), 
\end{align*}
where $\sigma (dx)$ denotes the $(N-1)$-dimensional Lebesgue measure 
on $\partial \R ^{N}_{+}$. 
It is known \cite{her,ddm} that 
\begin{align*}
 C^{*}_{N}=2\frac{\Gamma ^2(\frac{N}{4})}{\Gamma ^2(\frac{N-2}{4})}. 
\end{align*}
The constant 
$j_{\frac{N-3}{2},1}$ in the condition \eqref{;cond3} of \tref{;tmain3} asymptotically 
coincides with $C^{*}_{N}$; indeed, Stirling's formula and 
\eqref{;j1asym}
entail that 
\begin{align*}
 \lim _{N\to \infty }\frac{1}{N}C^{*}_{N}=
 \lim _{N\to \infty }\frac{1}{N}j_{\frac{N-3}{2},1}=
 \frac{1}{2}. 
\end{align*}
In view of the fact \eqref{;convdiv1}, we conjecture that 
\begin{align*}
  \E _{x}\left[ 
  \exp \left( c\int _{0}^{t}\frac{d\ln _{s}}{|\bd _{s}|^{}}\right) 
  \right] 
  \begin{cases}
   <\infty & \text{if }c\le C^{*}_{N}, \\
   =\infty & \text{if }c>C^{*}_{N}, 
  \end{cases}
\end{align*}
for any $x\in \R ^{N}_{+}$\,$(x\neq 0)$ and $t>0$. 
We also note that $C^{*}_{N}$ is equal to $C_{N,\al }$ given in 
\eqref{;opfrac}, with $\al =1$ and with $N$ replaced by $N-1$. 
We show a connection of the representation 
\eqref{;fktr} with $(N-1)$-dimensional (relativistic) $1$-stable 
process in \ssref{;ssa2}. 

\subsection{Proof of \tref{;tmain3}}\label{;ss42}

We proceed to the proof of \tref{;tmain3}. 
From now on, we fix $x=(x',x_{N})\in \R _{+}^{N}$ and $t>0$. As 
$u_{0}$ is continuous and $u_{0}\ge (\not\equiv)\,0$, we may assume 
that there exist $\ep _{0}>0$, a nonempty open disc 
$D\subset \R ^{N-1}$ and an interval 
$J=(l,r)\subset (0,\infty )$ $(l<r)$ such that 
\begin{align}\label{;ubound}
 u_{0}(y)\ge \ep _{0} \quad \text{for all }y\in D\times J. 
\end{align}
We fix an $\A \in (0,1/2)$ and set $\ga =1-2\A $ 
as in preceding sections. For each $n\in \N $ we set an 
event $A_{n}$ by 
\begin{align*}
 A_{n}=\left\{ 
 \max _{\A t\le s\le (1-\A )t}\left| \bd _{s}\right| <1/n, \ 
 \bd _{t}\in D
 \right\} . 
\end{align*}
Let $n_{0}\in \N $ be such that $\vi $ is nonincreasing on $(0,1/n_{0}]$. Then, 
for $n\ge n_{0}$, by restricting the $\P _{x}$-expectation to the event 
$A_{n}\cap \{ |\bn _{t}|\in J\} $ and using \eqref{;ubound}, 
the expectation \eqref{;fktr} is bounded from below by 
\begin{equation}\label{;qpmain1}
 \begin{split}
 &\ep _{0}\E _{x}\left[ 
 \exp \left\{ 
 \int _{\A t}^{(1-\A )t}V(\bd _{s},0)\,d\ln _{s}
 \right\} ;\,A_{n}\cap \{ |\bn _{t}|\in J\} 
 \right] \\
 &\ge \ep _{0}\P _{x}\left( A_{n}\right)
 \times \II _{n}, 
 \end{split}
\end{equation}
where 
\begin{align*}
 \II _{n}:=\E _{x}\left[ 
 \exp \left\{ 
 \vi \Bigl( \frac{1}{n}\Bigr) \left( 
 \ln _{(1-\A )t}-\ln _{\A t}
 \right) 
 \right\} ;\, |\bn _{t}|\in J
 \right] . 
\end{align*}
Here we used the independence of $\bd $ and 
$\bn $. Applying \pref{;pI1} with $N-1$ replacing $N$, we have the 
following estimate for $\P _{x}(A_{n})$: 
\begin{align}\label{;3e1}
 \P _{x}(A_{n})\ge \Cst \left( \frac{1}{n}\right) ^{N-1}
  \exp\left( 
  -\frac{1}{2}j_{\frac{N-3}{2} ,1}^2n^2\gamma t
  \right) \quad \text{for all }n\in \N , 
\end{align}
with some positive constant $\Cst $ independent of $n$. 
As to $\II _{n}$, we have 
\begin{prop}\label{;pI2}
 There exists a positive constant $\Cst '\equiv \Cst '(x_{N},t,\A ,J)$ 
 independent of $n$ such that 
 \begin{align*}
  \II _{n}\ge \Cst '\vi \Bigl( \frac{1}{n}\Bigr) 
  \exp \left\{ 
  \frac{1}{2}\vi ^{2}\Bigl( \frac{1}{n}\Bigr) \gamma t-2\vi  
  \Bigl( \frac{1}{n}\Bigr) 
  \right\} \quad \text{for all }n\in \N . 
 \end{align*}
\end{prop}

Combining these two estimates leads to \tref{;tmain3}: 
\begin{proof}[Proof of \tref{;tmain3}]
 By \eqref{;3e1}, \pref{;pI2} and the condition \eqref{;cond3}, 
 the right-hand side of \eqref{;qpmain1} diverges 
 as $n\to \infty $, which concludes the theorem. 
\end{proof}

It remains to prove \pref{;pI2}. 
For the rest of the section, we denote by 
the pair $(\{ B_{t}\} _{t\ge 0},\{ P_{x}\} _{x\in \R })$ a one-dimensional 
Brownian motion and by $\{ L_{t}\} _{t\ge 0}$ the local time process of 
$\{ B_{t}\} _{t\ge 0}$ at the origin, so that we may write 
\begin{align*}
 \II _{n}=E_{x_{N}}\left[ 
 \exp \left\{ 
 \vi \Bigl( \frac{1}{n}\Bigr) \left( L_{(1-\A )t}-L_{\A t}\right) 
 \right\} ;|B_{t}|\in J
 \right] . 
\end{align*}
Here $E_{x_{N}}$ denotes the expectation 
relative to $P_{x_{N}}$ as above. 

\begin{proof}[Proof of Proposition \ref{;pI2}] 
Restricting the $P_{x_{N}}$-expectation to the event $\{ |B_{\A t}|<1\} $ 
and using the Markov property, we have 
\begin{equation}\label{;q0pI2}
\begin{split}
 \II _{n}&\ge E_{x_{N}}\left[ 
 \psi (B_{\A t});\,|B_{\A t}|<1 
 \right] \\
 &=\int _{-1}^{1}dx\,
 g_{1}(\A t,x-x_{N})
 \psi (x), 
\end{split}
\end{equation}
where we set 
\begin{align*}
 \psi (x):=E_{x}\left[ 
 \exp \left\{ 
 \vi \Bigl( 
 \frac{1}{n}
 \Bigr) L_{\gamma t}
 \right\} ;\,|B_{(1-\A )t}|\in J
 \right] , \quad x\in \R . 
\end{align*}
Restricting the expectation to the event $\{ |B_{\gamma t}|<1\} $ in 
the definition of $\psi $, and using the Markov property again, we see that 
for every $x\in \R $, 
\begin{align}\label{;q1pI2}
 \psi (x)&\ge E_{x}\left[ 
 \exp \left\{ 
 \vi \Bigl( 
 \frac{1}{n}
 \Bigr) L_{\gamma t}
 \right\} 
 P_{B_{\gamma t}}\left( |B_{\A t}|\in J\right) ;\,
 |B_{\gamma t}|<1
 \right] \notag \\
 &\ge c_{1}E_{x}\left[ 
 \exp \left\{ 
 \vi \Bigl( 
 \frac{1}{n}
 \Bigr) L_{\gamma t}
 \right\} ;\,|B_{\gamma t}|<1
 \right] , 
\end{align}
where 
$c_{1}:=\inf \limits_{|z|\le 1}P_{z}\left( |B_{\A t}|\in J\right) >0$. 
We recall that for every $x\in \R $ and $s>0$, the joint distribution of 
$L_{s}$ and $B_{s}$ under $P_{x}$ is given by 
\begin{align}
 P_{x}\left( 
 L_{s}=0,B_s\in dz
 \right) &=
 \frac{1}{\sqrt{2\pi s}}\exp \left\{ 
 -\frac{(z-x)^2}{2s}
 \right\} \left\{ 
 1-\exp \left( 
 -\frac{2xz}{s}
 \right) 
 \right\} dz \notag  
\intertext{for $z\in \{ xz\ge 0\} $, and }
 P_{x}\left( 
 L_{s}\in dy,B_s\in dz
 \right) &=\frac{1}{\sqrt{2\pi s^3}}\left( y+|z|+|x|\right) 
 \exp \left\{ -\frac{\left( y+|z|+|x|\right) ^2}{2s}
 \right\} dydz  \label{;jloc}
\end{align}
for $y>0,z\in \R $; see \cite[p.155, Formula~1.3.8]{bs} and also 
Exercise~\thetag{3.8} in \cite[Chapter~XII]{rey}. Using 
this expression of the joint distribution, we see that 
the expectation in \eqref{;q1pI2} is estimated as, for all $|x|<1$, 
\begin{align*}
 &E_{x}\left[ 
 \exp \left\{ 
 \vi \Bigl( 
 \frac{1}{n}
 \Bigr) L_{\gamma t}
 \right\} ;\,|B_{\gamma t}|<1
 \right] \\
 &=\int _{-1}^{1}
 dz\,g_{1}(\ga t,z-x)
 \\
 &\quad +\frac{1}{2}\vi \Bigl( \frac{1}{n}\Bigr) \int _{-1}^{1}dz\,
 \exp \left\{ 
 \frac{1}{2}\vi ^{2}\Bigl( \frac{1}{n}\Bigr) \gamma t-
 \vi \Bigl( \frac{1}{n}\Bigr) \left( |z|+|x|\right) 
 \right\} \mathrm{Erfc}\left( 
 \frac{|z|+|x|}{\sqrt{2\gamma t}}-\vi \Bigl( \frac{1}{n}\Bigr) 
 \sqrt{\frac{\gamma t}{2}}
 \right) \\
 &\ge \vi \Bigl( \frac{1}{n}\Bigr) 
 \exp \left\{ 
 \frac{1}{2}\vi ^{2}\Bigl( \frac{1}{n}\Bigr) \gamma t-
 2\vi \Bigl( \frac{1}{n}\Bigr) 
 \right\} 
 \mathrm{Erfc}\left( 
 \sqrt{\frac{2}{\gamma t}}
 \right) 
\end{align*}
with 
\begin{align*}
 \mathrm{Erfc}(z)=\frac{2}{\sqrt{\pi }}\int _{z}^\infty e^{-y^2}\,dy, 
 \quad z\in \R . 
\end{align*}
For the first equality in the above estimate, refer also to 
\cite[p.155, Formula~1.3.7]{bs}. Combining this estimate with \eqref{;q1pI2}, 
we see from \eqref{;q0pI2} that 
\begin{align*}
 \II _{n}\ge c_{1}c_{2}\vi \Bigl( \frac{1}{n}\Bigr) 
 \exp \left\{ 
 \frac{1}{2}\vi ^{2}\Bigl( \frac{1}{n}\Bigr) \gamma t-
 2\vi \Bigl( \frac{1}{n}\Bigr) 
 \right\} , 
\end{align*}
where 
\begin{align*}
 c_{2}:=\mathrm{Erfc}\left( 
 \sqrt{\frac{2}{\gamma t}}
 \right) \int _{-1}^{1}dx\,
 g_{1}(\A t,x-x_{N}). 
\end{align*}
The proof is complete. 
\end{proof}

\appendix 
\section*{Appendix}
\renewcommand{\thesection}{A}
\setcounter{equation}{0}
\setcounter{prop}{0}
\setcounter{lem}{0}
\setcounter{rem}{0}

\subsection{Proof of \lref{;keylem}}\label{;ssa1}

In this subsection we give a proof of \lref{;keylem}. 
For every $\mu >-1$, we denote by 
\begin{align*}
 0<j_{\mu ,1}<\cdots <j_{\mu ,k}<\cdots 
\end{align*}
the positive zeros of $J_{\mu }$. It is known that 
\begin{align*}
 j_{\mu ,k}=\left( k+\frac{1}{2}\mu -\frac{1}{4}\right) \pi +O\left( \frac{1}{k}\right) \quad \text{as }k\to \infty 
\end{align*}
when $\mu \neq \pm 1/2$; see, e.g., \cite[p.506]{wat}. 
Recall also $J_{1/2}(z)=\sqrt{2/(\pi z)}\sin z$, 
$J_{-1/2}(z)=\sqrt{2/(\pi z)}\cos z$. To prove the lemma, 
we need the following: 
\begin{lem}\label{;l2pI1}
 For $\mu >-1/2$, it holds that 
 \begin{align*}
  \lim _{k\to \infty }\sqrt{\frac{\pi j_{\mu ,k}}{2}}\left| 
  J_{\mu +1}(j_{\mu ,k})
  \right| =1. 
 \end{align*}
\end{lem}

\begin{proof}
By the asymptotic expansion 
\cite[Equation~\thetag{5.11.6}]{leb} 
of $J_{\mu }$ with $\mu >-1/2$, 
for any $\ep \in (0,1)$, there exists an $L>0$ such that 
for all $z>L$, both 
\begin{align*}
 \left| 
 \sqrt{\frac{\pi z}{2}}J_{\mu }(z)-\cos \left( 
 z-\frac{1}{2}\mu \pi -\frac{1}{4}\pi 
 \right) 
 \right| &<\ep 
\intertext{and }
 \left| 
 \sqrt{\frac{\pi z}{2}}J_{\mu +1}(z)-\cos \left\{ 
 z-\frac{1}{2}(\mu +1)\pi -\frac{1}{4}\pi 
 \right\} 
 \right| &<\ep 
\end{align*}
hold. Then, for all $k$ such that $j_{\mu ,k}>L$, we have 
\begin{align*}
 \left| 
 \cos \left( 
 j_{\mu ,k}-\frac{1}{2}\mu \pi -\frac{1}{4}\pi 
 \right) 
 \right| <\ep && \text{and} && 
 \left| 
 \sqrt{\frac{\pi j_{\mu ,k}}{2}}J_{\mu +1}(j_{\mu ,k})-\sin \left( 
 j_{\mu ,k}-\frac{1}{2}\mu \pi -\frac{1}{4}\pi 
 \right) 
 \right| <\ep . 
\end{align*}
Therefore, for sufficiently large $k$, 
\begin{align*}
 \sqrt{1-\ep ^2}-\ep <
 \sqrt{\frac{\pi j_{\mu ,k}}{2}}\left| J_{\mu +1}(j_{\mu ,k})\right| 
 <1+\ep , 
\end{align*}
from which the assertion of the lemma follows. 
\end{proof}

We are in a position to prove Lemma \ref{;keylem}. 
For every positive integer $N$, set $\mu =(N-2)/2$. 

\begin{proof}[Proof of Lemma \ref{;keylem}] 
As it is known \cite[Section~8]{ken}, \cite[p.373, Formula~1.1.4]{bs} that 
\begin{align}\label{;a1e1}
 P_{\xi }\left( 
 \max _{0\le s\le T}|B_{s}|<1
 \right) 
 =\frac{2}{|\xi |^{\mu }}\sum _{k=1}^\infty 
 \frac{J_{\mu }\left( j_{\mu ,k}|\xi |\right) }{j_{\mu ,k}J_{\mu +1}(j_{\mu ,k})}\exp \left( -\frac{1}{2}j_{\mu ,k}^2T\right) 
\end{align}
for all $|\xi |<1$, 
we have 
\begin{equation}\label{;q4pI1}
\begin{split}
&\int _{|\xi |<1}d\xi \,P_{\xi }\left( 
 \max _{0\le s\le T}|B_{s}|<1
 \right) \\
 &=2\varpi _{N}\int _{0}^{1}dr\,
 r^{\mu +1}\sum _{k=1}^\infty 
 \frac{J_{\mu }\left( j_{\mu ,k}r\right) }{j_{\mu ,k}J_{\mu +1}(j_{\mu ,k})}\exp \left( -\frac{1}{2}j_{\mu ,k}^2T\right) . 
\end{split}
\end{equation}
First we consider the case $\mu \ge 0$ (i.e., $N\ge 2$). 
By Lemma \ref{;l2pI1} and by the fact that $J_{\mu }$ is a bounded function for $\mu \ge 0$, 
we see that the series in the integrand relative to $r$ converges uniformly 
on the interval $[0,1]$, hence the termwise integration is possible. 
By the relation 
$\left\{ z^{\mu +1}J_{\mu +1}(z)\right\} '=z^{\mu +1}J_{\mu }(z)$, 
we have 
\begin{align*}
 \int _{0}^{1}r^{\mu +1}J_{\mu }(j_{\mu ,k}r)\,dr
 =\frac{J_{\mu +1}(j_{\mu ,k})}{j_{\mu ,k}}. 
\end{align*}
Therefore the right-hand side of \eqref{;q4pI1} is equal to 
\begin{align*}
 2\varpi _{N}\sum _{k=1}^\infty 
 \frac{1}{j_{\mu ,k}^2}\exp \left( -\frac{1}{2}j_{\mu ,k}^2T\right) , 
\end{align*}
which yields the lemma for $N\ge 2$.  
By writing down the right-hand side 
of \eqref{;a1e1} into 
\begin{align*}
 \frac{4}{\pi }
 \sum _{k=1}^{\infty }
 (-1)^{k-1}
 \frac{\cos \left( \frac{2k-1}{2}\pi \xi \right) }{2k-1}
 \exp \left\{ 
 -\frac{\pi ^2}{8}(2k-1)^2T
 \right\} 
\end{align*}
for $\mu =-1/2$, the case $N=1$ is similarly proved. 
\end{proof}

\subsection{A connection of \eqref{;fktr} with $1$-stable processes}
\label{;ssa2}

In this subsection we explore a connection of the Feynman-Kac representation 
\eqref{;fktr} with $1$-stable processes. For ease of exposition, 
we start the one-dimensional Brownian motion $\bn $ from the 
origin, that is, we consider the expression \eqref{;fktr} on the boundary 
$\partial \R ^{N}_{+}$, with which we define the function $u:[0,\infty )\times \R ^{N-1}\to [0,\infty )$ 
by 
\begin{align}\label{;fkd}
 u(t,x)=
 \E _{(x,0)}\left[ 
 u_{0}(\bd _{t},\left| \bn _{t}\right| )\exp 
 \left\{ 
 \int _{0}^{t}V(\bd _{s})\,d\ln _{s}
 \right\} 
 \right] . 
\end{align}
Here and below we regard $V:\partial \R ^{N}_{+}\to \R $ as a function on 
$\R ^{N-1}$ and simply write $V(x,0)=V(x)$ for 
$(x,0)\in \partial \R ^{N}_{+}$. 

For every real-valued continuous function $w$ on $[0,\infty )$ vanishing 
at the origin, we write 
\begin{align*}
 \ovl{w}_{t}=\max _{0\le s\le t}w_{s}, \quad t\ge 0, 
\end{align*}
and denote by $\tau _{\cdot }(w)$ the right-continuous inverse of $\ovl{w}$: 
\begin{align*}
 \tau _{a}(w)=\inf \left\{ 
 t>0;\,\ovl{w}_{t}>a
 \right\} ,\quad a\ge 0. 
\end{align*}
Let $\{ \be _{t}\} _{t\ge 0}$ together with a probability measure $P$, 
be a one-dimensional standard Brownian motion and 
$(\{ \W (t)\} _{t\ge 0},\{ Q_{x}\} _{x\in \R ^{N-1}})$ an 
$(N-1)$-dimensional Brownian motion. We assume that these two 
processes are defined on distinct measurable spaces. 
By the equivalence in law between $L^{N}$ and 
$\ovl{\be }$ due to L\'evy, we have the following identity 
as to the additive functional in \eqref{;fkd}: 
\begin{align*}
 \int _{0}^{\cdot }
 V(\bd _{s})\,dL^{N}_{s}
 \stackrel{(d)}{=}\int _{0}^{\cdot }V(\W (s))\,d\ovl{\be }_{s}, 
\end{align*}
where in the right-hand side, the law is with respect to 
the product probability measure $\qp{x}$. We make 
the change of variables with $s=\tau _{a}(\be )$ to see that 
for all $t\ge 0$, 
\begin{align}\label{;tch}
 \int _{0}^{t}V(\W (s))\,d\ovl{\be }_{s}
 =\int _{0}^{\ovl{\be }_{t}}V\left( \W (\tau _{a}(\be ))\right) da. 
\end{align}
It is well known that the process $\{ \W (\tau _{a}(\be ))\} _{a\ge 0}$ 
has the same law as a rotationally invariant $1$-stable process 
(or Cauchy process) starting from $x$; indeed, for every $a\ge 0$ and 
$\xi \in \R ^{N-1}$, 
\begin{equation}\label{;cauchy}
 \begin{split}
 &\qp{x}
 \left[ 
 \exp \left\{ 
 i\xi \cdot \left( W(\tau _{a}(\be ))-x\right) 
 \right\} 
 \right] \\
 &=P\left[ 
 \exp \left\{ 
 -\frac{1}{2}|\xi |^{2}\tau _{a}(\be )
 \right\} 
 \right] \\
 &=\exp \left( -a|\xi |\right) , 
 \end{split}
\end{equation}
where the last equality follows from the fact 
\begin{align*}
 P(\tau _{a}(\be )\in ds)
 =\frac{a}{\sqrt{2\pi s^{3}}}\exp \left( 
 -\frac{a^2}{2s}
 \right) ds, \quad s>0, 
\end{align*}
when $a>0$. In \eqref{;cauchy} and in the remainder of this section, 
for any probability measure $\mu $, 
the notation $\mu [\,\cdot \,]$ stands for the expectation with 
respect to $\mu $. 

The connection will be clearer if we take the Laplace transform 
of \eqref{;fkd} in variable $t$. Given a positive real $m$, 
let $(\{ \Xm _{t}\} _{t\ge 0},\{ \pr _{x}\} _{x\in \R ^{N-1}})$ be 
an $(N-1)$-dimensional relativistic $1$-stable process with 
mass $m$, that is, under $\pr _{x}$, the process $\Xm -x$ is a 
L\'evy process with characteristic function 
\begin{align}\label{;chfrel}
 \ex _{x}\left[ 
 \exp \left\{ 
 i\xi \cdot (\Xm _{t}-x)
 \right\} 
 \right] =\exp \left\{ 
 -t\left( 
 \sqrt{|\xi |^{2}+m^2}-m
 \right) 
 \right\} , 
 \quad t\ge 0,\, \xi \in \R ^{N-1}. 
\end{align}
The infinitesimal generator of $\Xm $ is the relativistic Schr\"odinger operator 
$m-\sqrt{-\Delta +m^2}$ (cf.~\cite{cms}). 
For each $x\in \R ^{N-1}$, set  
 \begin{align*}
  u_{m}(x)
  :=\int _{0}^{\infty }dt\,
  e^{-\frac{1}{2}m^2t}u(t,x). 
 \end{align*}
Then the function $u_{m}$ is related with the process $\Xm $ in the following 
fashion: 
\begin{prop}\label{;prelat}
 It holds that for all $x\in \R ^{N-1}$, 
 \begin{align}\label{;lt}
 u_{m}(x)=
 \int _{0}^\infty dt\,e^{-mt}
 E_{x}\left[ 
 f_{m}(\Xm _{t})\exp \left\{ 
 \int _{0}^{t}V(\Xm _{s})\,ds
 \right\} 
 \right] , 
\end{align}
where $f_{m}:\R ^{N-1}\to [0,\infty ) $ is given by 
\begin{align*}
 f_{m}(x)=\int _{0}^{\infty }dt\,e^{-\frac{1}{2}m^2t}f_{0}(t,x) 
\end{align*}
with 
\begin{align*}
f_{0}(t,x):=
\int _{\R ^{N-1}}dz\,g_{N-1}(t,z-x)
\int _{\R }\frac{dy}{t}\,|y|g_{1}(t,y)u_{0}(z,|y|), 
\quad t>0,\,x\in \R ^{N-1}. 
\end{align*}
\end{prop}

For the Brownian motion $\be $ introduced above, we denote its local time 
at level $0$ by $\{ L_{t}\} _{t\ge 0}$, to which we associate the 
measure $\mu _{L}$ on $(0,\infty )$ via 
\begin{align*}
 \mu _{L}((a,b])&:=P \left[ L_{b}\right] -P \left[ L_{a}\right] \\
 &=\int _{a}^{b}\frac{ds}{\sqrt{2\pi s}}
\end{align*}
for all $0<a<b$. For each $v>0$ and $y\in \R $, we denote by 
$P_{v,y}$ the regular version of conditional probability 
$P(\,\cdot \,|\,\be _{v}=y)$, namely under $P_{v,y}$, the process 
$\{ \be _{s}\} _{0\le s\le v}$ is a Brownian bridge over 
$[0,v]$ starting from $0$ and ending at $y$. From now on, 
we fix $x\in \R ^{N-1}$. We start the proof of \pref{;prelat} with the 
following lemma: 

\begin{lem}\label{;lrewr}
 It holds that for every $t>0$, 
 \begin{align*}
  u(t,x)=
  \int _{0}^{t}\mu _{L}(dv)\,
  Q_{x}\otimes P_{v,0}\left[ 
  f_{0}(t-v,W(v))\exp 
  \left( 
  \int _{0}^{v}V(W(s))\,dL_{s}
  \right) 
  \right] . 
 \end{align*}
\end{lem}

In order to prove this lemma, we recall some facts on the path decomposition 
of Brownian motion at the last zero before a fixed time. For every given 
$t>0$, we set 
\begin{align*}
 \ga _{t}=\sup \left\{ 
 s\le t;\,\be _{s}=0
 \right\} . 
\end{align*}
Then it holds that conditionally on $\ga _{t}=v\ (0<v<t)$: 
\begin{enumerate}[(i)]{}
 \item $\{ \be _{s}\} _{0\le s\le v}$ is identical in law with 
 a Brownian bridge $\{ b_{s}\} _{0\le s\le v}$ such that 
 $b_{0}=b_{v}=0$; 
 
 \item $\{ \be _{s+v}\} _{0\le s\le t-v}$ is identical in law with 
 \begin{align*}
  \{ \mathbf{n}M_{s}\} _{0\le s\le t-v}, 
 \end{align*}
 where $\mathbf{n}$ is a Bernoulli distributed random variable 
 with parameter $1/2$ and $M$ is a Brownian meander of 
 duration $t-v$, 
\end{enumerate}
with these three elements $b,\mathbf{n},M$ being independent. 
It is also known that $\ga _{t}$ follows the arcsine law: 
 \begin{align*}
  P(\ga _{t}\in dv)=
  \frac{dv}{\pi \sqrt{v(t-v)}}, \quad v\in (0,t). 
 \end{align*}
For descriptions of the decomposition, see \cite[Section~3.1]{may} 
and references therein. 

\begin{proof}[Proof of \lref{;lrewr}]
 By the equivalence in law and by the fact that 
 the local time $L$ does not increase when $\be $ is away from $0$, 
 we may write 
 \begin{align*}
  u(t,x)&=\qp{x}
  \left[ 
  u_{0}(W(t),|\be _{t}|)\exp 
  \left( 
  \int _{0}^{t}V(W(s))\,dL_{s}
  \right) 
  \right] \\
  &=\qp{x}
  \left[ 
  u_{0}(W(t),|\be _{t}|)\exp 
  \left( 
  \int _{0}^{\ga _{t}}V(W(s))\,dL_{s}
  \right) 
  \right] , 
 \end{align*}
 which is rewritten, by using the above facts and the Markov property of 
 $W$, as 
 \begin{align}\label{;q2a2}
  \int _{0}^{t}\frac{dv}{\pi \sqrt{v(t-v)}}\,
  Q_{x}\Biggl[ 
  P_{v,0}&\left[ 
  \exp 
  \left( 
  \int _{0}^{v}V(W(s))\,dL_{s}
  \right) 
  \right] \notag \\
  &\ \times 
  \qp{W(v)}\left[ 
  u_{0}\left( W(t-v),\left| \mathbf{n}M_{t-v}\right| \right) 
  \right] 
  \Biggr] . 
 \end{align}
 Since 
 \begin{align*}
  P(M_{t-v}\in dy)=\sqrt{\frac{2\pi }{t-v}}\,
  yg_{1}(t-v,y)\,dy, \quad y>0, 
 \end{align*}
 we have in \eqref{;q2a2} 
 \begin{align*}
  &\qp{W(v)}\left[ 
  u_{0}\left( W(t-v),\left| \mathbf{n}M_{t-v}\right| \right) 
  \right] \\
  &=\sqrt{\frac{\pi }{2(t-v)}}
  \int _{\R ^{N-1}}dz\,g_{N-1}\left( t-v,z-W(v)\right) 
  \int _{-\infty }^{\infty }dy\,|y|g_{1}(t-v,y)u_{0}(z,|y|)\\
  &=\sqrt{\frac{\pi (t-v)}{2}}f_{0}\left( t-v,W(v)\right) 
 \end{align*}
 by the definition of $f_{0}$. Plugging this into 
 \eqref{;q2a2}, we obtain the claimed representation for $u(t,x)$. 
\end{proof}

Using \lref{;lrewr}, we prove \pref{;prelat}. To this end, we set 
$\btm _{t}=\be _{t}+mt,\,t\ge 0$, and recall the identity in law: 
\begin{align}\label{;relrel}
 \left( 
 \{ \Xm _{t}\} _{t\ge 0},\pr _{x}
 \right) 
 \stackrel{(d)}{=}
 \bigl( 
 \left\{ 
 W(\tau _{t}(\btm ))
 \right\} _{t\ge 0},\qp{x}
 \bigr) , 
\end{align}
which can easily be checked by similar calculation to 
\eqref{;cauchy}, upon using the Cameron-Martin relation; indeed, 
for every $t\ge 0$ and $\xi \in \R ^{N-1}$, 
\begin{align*}
 &\qp{x}\left[ 
 \exp \left\{ 
 i\xi \cdot \left( 
 W(\tau _{t}(\btm ))-x
 \right) 
 \right\} 
 \right] \\
 &=\qp{x}\left[ 
 \exp \left( 
 mt-\frac{1}{2}m^2\tau _{t}(\be )
 \right) \exp \left\{ 
 i\xi \cdot \left( 
 W(\tau _{t}(\be ))-x
 \right) 
 \right\} 
 \right] \\
 &=P\left[ 
 \exp \left\{ 
 mt-\frac{1}{2}\left( |\xi |^{2}+m^2\right) \tau _{t}(\be )
 \right\} 
 \right] \\
 &=\exp \left\{ 
 t\left( 
 m-\sqrt{|\xi |^{2}+m^{2}}
 \right) 
 \right\},  
\end{align*}
in agreement with \eqref{;chfrel}. 
We are in a position to prove \pref{;prelat}. 
\begin{proof}[Proof of \pref{;prelat}]
 By \eqref{;relrel}, we rewrite the $\pr _{x}$-expectation in the 
 right-hand side of \eqref{;lt} as 
 \begin{align*}
  &\qp{x}\left[ 
  f_{m}\left( W(\tau _{t}(\btm ))\right) 
  \exp 
  \left\{ 
  \int _{0}^{t}V\left( W(\tau _{s}(\btm ))\right) ds
  \right\} 
  \right] \\
  &=\qp{x}
  \left[ 
  \exp \left( 
  mt-\frac{1}{2}m^2\tau _{t}(\be )
  \right) 
  f_{m}\left( W(\tau _{t}(\be ))\right) 
  \exp 
  \left\{ 
  \int _{0}^{t}V\left( W(\tau _{s}(\be ))\right) ds
  \right\} 
  \right] , 
 \end{align*}
 where for the second line, we used the Cameron-Martin relation 
 under $P$. Hence by Fubini's theorem, the right-hand side of 
 \eqref{;lt} is equal to 
 \begin{align*}
  \qp{x}
  \left[ \int _{0}^{\infty }dt\,
  \exp \left( 
  -\frac{1}{2}m^2\tau _{t}(\be )
  \right) 
  f_{m}\left( W(\tau _{t}(\be ))\right) 
  \exp 
  \left\{ 
  \int _{0}^{t}V\left( W(\tau _{s}(\be ))\right) ds
  \right\} 
  \right] . 
 \end{align*}
 By changing variables with $t=\ovl{\be }_{v}$ and noting 
 \eqref{;tch}, the above expression is further rewritten as 
 \begin{align}
  &\qp{x}
  \left[ 
  \int _{0}^{\infty }d\ovl{\be }_{v}\,
  e^{-\frac{1}{2}m^2v}
  f_{m}(W(v))
  \exp 
  \left( 
  \int _{0}^{v}V(W(s))\,d\ovl{\be }_{s}
  \right) 
  \right] \notag \\
  &=\qp{x}
  \left[ 
  \int _{0}^{\infty }dL_{v}\,
  e^{-\frac{1}{2}m^2v}
  f_{m}(W(v))
  \exp 
  \left( 
  \int _{0}^{v}V(W(s))\,dL_{s}
  \right) 
  \right] \notag \\
  &=\int _{0}^{\infty }\mu _{L}(dv)\,e^{-\frac{1}{2}m^2v}
  \qp{x}_{v,0}
  \left[ 
  f_{m}(W(v))
  \exp 
  \left( 
  \int _{0}^{v}V(W(s))\,dL_{s}
  \right) 
  \right] , \label{;q1a2}
 \end{align}
 where the first equality is due to L\'evy's equivalence, and the 
 second follows from the definition of $\mu _{L}$ and the fact that $dL_{v}$ is 
 carried by the set $\{ v\ge 0;\,\be _{v}=0\} $; for the validity of the latter 
 computation,  refer to Exercise~\thetag{2.29} in \cite[Chapter~VI]{rey} 
 (closely related is 
 the theory of Brownian excursions, see Chapter~XII of the same reference). 
 By the definition of $f_{m}$, 
 we may write 
 \begin{align*}
  f_{m}(W(v))=\int _{v}^{\infty }dt\,
  e^{-\frac{1}{2}m^2(t-v)}f_{0}(t-v,W(v)). 
 \end{align*}
 Inserting this expression into \eqref{;q1a2} and using Fubini's theorem, 
 we see that \eqref{;q1a2} is equal to 
 \begin{align*}
  \int _{0}^{\infty }dt\,
  e^{-\frac{1}{2}m^2t}
  \int _{0}^{t}\mu _{L}(dv)\,
  Q_{x}\otimes P_{v,0}\left[ 
  f_{0}(t-v,W(v))\exp 
  \left( 
  \int _{0}^{v}V(W(s))\,dL_{s}
  \right) 
  \right] , 
 \end{align*}
 which agrees with $u_{m}(x)$ by \lref{;lrewr}. This ends the proof of 
 the proposition. 
\end{proof}

\begin{rem} 
\thetag{1}~A point of the above computation is the nonnegativity of $u_{0}$, 
which allows us to use Fubini's theorem without taking the integrability 
into account, and hence we may take $u_{0}\equiv 1$ to obtain 
for all $x \in \R ^{N-1}$, 
\begin{align*}
 \frac{m^2}{2}\int _{0}^{\infty }dt\,e^{-\frac{1}{2}m^2t}
 \E _{(x,0)}&\left[ 
 \exp 
 \left\{ 
 \int _{0}^{t}V(\bd _{s})\,d\ln _{s}
 \right\} 
 \right] \\
 &=m
 \int _{0}^\infty dt\,e^{-mt}
 E_{x}\left[ 
 \exp \left\{ 
 \int _{0}^{t}V(\Xm _{s})\,ds
 \right\} 
 \right] . 
\end{align*}

 \noindent 
 \thetag{2}~If we take $x=(x',x_{N})$ with $x_{N}>0$ in \eqref{;fktr}, 
 then its Laplace transform admits the following representation: 
 \begin{align}
 &\int _{0}^{\infty }dt\,e^{-\frac{1}{2}m^2t}
  \E _{x}\left[ 
 u_{0}(\bd _{t},|\bn _{t}|)\exp 
 \left\{ 
 \int _{0}^{t}V(\bd _{s})\,d\ln _{s}
 \right\} 
 \right] \notag \\
 &=m^{N}\int_{\R ^{N-1}}dz\,
  \biggl\{ 
  x_{N}\Phi _{N}\Bigl(  
  m\sqrt{|z-x'|^2+x_{N}^2}
  \Bigr) u_{m}(z) \notag \\
  &\qquad \qquad \qquad \qquad +\int _{0}^{\infty }dr\,u_{0}(z,r)
  \int _{|r-x_{N}|}^{r+x_{N}}
  d\eta \,\eta \Phi _{N}\left( 
  m\sqrt{|z-x'|^2+\eta ^2}
  \right)  
  \biggr\} , \label{;ltgen}
 \end{align}
 where we set 
 \begin{align*}
  \Phi _{N}(y)=\frac{2}{\sqrt{(2\pi y)^{N}}}K_{\frac{N}{2}}(y), \quad y>0, 
 \end{align*}
 with $K_{\frac{N}{2}}$ the modified Bessel function of the third kind 
 of index $N/2$, and $u_{m}$ is defined as above and 
 expressed as \eqref{;lt}. 
 The representation \eqref{;ltgen} is seen by decomposing 
 \eqref{;fktr} into the sum 
 \begin{align}\label{;decompA}
 \E _{x}\biggl[ 
 u_{0}(\bd _{t},|\bn _{t}|)\exp 
 \biggl\{ 
 \int _{\sigma ^{N}_{0}}^{t}V(\bd _{s})\,d\ln _{s}
 \biggr\} ;\,\sigma ^{N}_{0}\le t
 \biggr] +
 \E _{x}\left[ 
 u_{0}(\bd _{t},|\bn _{t}|);\,\sigma ^{N}_{0}>t
 \right] , 
\end{align}
where $\sigma ^{N}_{0}$ is the first hitting time of $\bn $ to the origin. 
By conditioning on $\sigma ^{N}_{0}$ and using the (strong) Markov property 
of Brownian motion, we may see that the first term of \eqref{;decompA} 
is rewritten as 
\begin{align*}
 \int _{0}^{t}dv\,
 \frac{x_{N}}{\sqrt{2\pi v^{3}}}
 \exp \left( -\frac{x_{N}^2}{2v}\right) 
 \int _{\R ^{N-1}}\frac{dz}{\sqrt{(2\pi v)^{N-1}}}
 \exp \left\{ 
 -\frac{|z-x'|^{2}}{2v}
 \right\} u(t-v,z)
\end{align*}
with $u$ the function defined by \eqref{;fkd}. We use the explicit 
representation of the transition density of one-dimensional Brownian motion 
absorbed at the origin (see, e.g., \cite[Problem~2.8.6]{ks}) to 
rewrite the second term of \eqref{;decompA} as 
\begin{align*}
 \int _{\R ^{N-1}}dz\,\int _{0}^{\infty }dr\,
 u_{0}(z,r)\int _{|r-x_{N}|}^{r+x_{N}}
 \frac{d\eta }{\sqrt{(2\pi t)^{N}}}\frac{\eta }{t}
 \exp \left( -\frac{|z-x'|^{2}+\eta ^2}{2t}\right) . 
\end{align*}
Combining these expressions and noting the relation that 
\begin{align*}
 \int _{0}^{\infty }dt\,t^{-\frac{N}{2}-1}
 \exp \left( -\frac{1}{2}m^2t-\frac{a^2}{2t}\right) 
 =2\left( \frac{m}{a}\right) ^{\frac{N}{2}}K_{\frac{N}{2}}(am)
\end{align*}
for any $a>0$ (cf.\ \cite[Equation~\thetag{5.10.25}]{leb}), 
we obtain \eqref{;ltgen}. 
\end{rem}
\medskip 

\noindent 
{\bf Acknowledgements.} The authors would like to thank an anonymous 
referee for carefully reading the manuscript and providing them 
with valuable comments.


\end{document}